\documentclass[11pt,letterpaper]{article}

\usepackage{amsmath,amssymb,amsthm}
\usepackage{array}
\usepackage[top=1in,bottom=1in,left=1in,right=1in]{geometry}
\usepackage{booktabs}
\usepackage{pgfplots}
\pgfplotsset{compat=1.17}
\usepackage{hyperref}
\hypersetup{colorlinks=true,linkcolor=blue,citecolor=blue,urlcolor=blue}

\newtheorem{theorem}{Theorem}[section]
\newtheorem{lemma}[theorem]{Lemma}
\newtheorem{corollary}[theorem]{Corollary}

\theoremstyle{definition}
\newtheorem{definition}[theorem]{Definition}
\newtheorem{example}[theorem]{Example}
\theoremstyle{remark}
\newtheorem{remark}[theorem]{Remark}

\newcommand{\C}{\mathbb{C}}
\newcommand{\R}{\mathbb{R}}

\title{\textbf{On the Roots of Connected Domination Polynomials}}

\author{Pingping Shao \and Chengye Zhao\thanks{College of Sciences, China Jiliang
University, Hangzhou 310018, Zhejiang, China. E-mail: s24080701012@cjlu.edu.cn (P.~Shao); cyzhao@cjlu.edu.cn (C.~Zhao, corresponding author).}}

\date{}

\begin{document}
\maketitle

\begin{abstract}
\noindent We determine the closure of the connected domination roots.
The main tool is a substitution formula for
the lexicographic product with a complete graph,
$D_c(G[K_n],x)=D_c(G,(x+1)^n-1)$, proved in Theorem~\ref{thm:cd}. Combined
with two explicit families of seed roots---the real roots of the cycles $C_n$
and the real roots of the joins $C_m\vee C_n$ lying in $(-1,0)$---this formula
gives the two main results: the closure of the real connected domination roots
is $(-\infty,0]$, and the closure of all connected domination roots is the
whole complex plane. These
are the connected domination analogues of the root-density theorems of
Brown--Tufts and Brown--Beaton for the ordinary domination polynomial.


\end{abstract}

\section{Introduction}

\subsection{Basic concepts}

All graphs $G=(V,E)$ considered in this paper are finite, undirected and simple.
For a vertex $v\in V$, its open neighbourhood is $N(v)=\{u:uv\in E\}$ and its
closed neighbourhood is $N[v]=N(v)\cup\{v\}$.

\begin{definition}[Dominating set and domination number]
A set $S\subseteq V$ is a \emph{dominating set} of $G$ if every vertex
$v\in V\setminus S$ is adjacent to at least one vertex of $S$, that is,
$N[S]=V$. The minimum cardinality of a dominating set is the \emph{domination
number}, denoted by $\gamma(G)$.
\end{definition}

\begin{definition}[Connected dominating set]
Let $G$ be a connected graph. A dominating set $S$ of $G$ is a \emph{connected
dominating set} if the induced subgraph $G[S]$ is connected; the minimum
cardinality of a connected dominating set is the \emph{connected domination
number}, denoted by $\gamma_c(G)$.
\end{definition}

\begin{definition}[Domination polynomial]
The \emph{domination polynomial} of a graph $G$ is
\begin{equation}
D(G,x)=\sum_{i=\gamma(G)}^{|V|} d(G,i)\,x^i,
\end{equation}
where $d(G,i)$ is the number of dominating sets of $G$ of cardinality $i$. For a
connected graph $G$ the \emph{connected domination polynomial} is defined
analogously by $D_c(G,x)=\sum_k d_c(G,k)x^k$, where $d_c(G,k)$ is the number of
connected dominating sets of $G$ of cardinality $k$. The roots of these two
polynomials are called \emph{domination roots} and \emph{connected domination
roots}, respectively.
\end{definition}

\begin{definition}[Graph substitution / lexicographic product]
Let $G$ and $H$ be graphs. The \emph{lexicographic product} (also called graph
substitution) $G[H]$ is obtained by replacing every vertex $v$ of $G$ by a copy
$H^v$ of $H$, and joining every vertex of $H^u$ to every vertex of $H^v$ if and
only if $uv\in E(G)$. Its vertex set is $V(G)\times V(H)$, and
$(u,i)(v,j)\in E(G[H])$ if and only if $uv\in E(G)$, or $u=v$ and
$ij\in E(H)$. Throughout this paper we take $H=K_n$, the complete graph of order
$n$, so that each copy $K_n^v$ is a clique of order $n$.
\end{definition}

\subsection{A brief survey of domination polynomials}

Domination theory is a classical branch of graph theory; the monograph of
Haynes, Hedetniemi and Slater~\cite{hhs} gives a systematic account of the state
of the art at the end of the twentieth century. The counting generating function
of dominating sets, the domination polynomial, was introduced by Arocha and
Llano in 2000~\cite{al}, and its basic theory was developed by
Alikhani and Peng~\cite{ap}. Akbari, Alikhani and Peng proved that several
families of graphs are uniquely determined by their domination
polynomials~\cite{aap}, and Kotek et al.\ established recurrence relations and
cut-vertex splitting formulas for the domination polynomial~\cite{kotek}.

The location of roots is a central theme in the theory of graph
polynomials. Sokal proved in 2004 that the chromatic
roots are dense in the whole complex plane~\cite{sokal}; Brown, Hickman and
Nowakowski proved the analogous statement for the roots of independence
polynomials~\cite{bhn}. The foundational work on domination roots is the 2014
paper of Brown and Tufts~\cite{bt}; its main results are: (i) the substitution lemma
(Lemma~\ref{lem:bt} below); (ii) the construction of explicit families of graphs
having domination roots in the right half-plane; (iii) the determination of the
limit curve of the domination roots of the complete bipartite graphs $K_{n,n}$;
and (iv) the proof that the closure of the domination roots is the whole complex
plane. Alikhani~\cite{alikhani} also studied the lexicographic product in this
framework and constructed families of graphs whose domination roots are dense in
the complex plane; he obtained the substitution formula for $K_n$
independently~\cite{isrn2013}, in a paper that appeared slightly later
than~\cite{bt}. Subsequently, Oboudi proved that every domination root satisfies
$|z+1|\le\sqrt[\delta+1]{2^n-1}$, and studied graphs whose domination
roots are all real or all integral~\cite{oboudi}; Brown and Beaton proved
that the closure of the real domination roots is exactly $(-\infty,0]$~\cite{bb};
Beaton and Brown investigated the unimodality of domination
polynomials~\cite{bb2}; and Akbari and Oboudi showed that cycles are determined
by their domination polynomials~\cite{ao}.

The connected domination number was introduced by Sampathkumar and Walikar in
1979~\cite{sw}. The corresponding counting polynomial appeared only recently: the
connected domination polynomial was introduced by Dhananjaya Murthy et
al.~\cite{murthy}, who computed it for paths, cycles and joins of graphs;
Mojdeh and Emadi studied the connected domination polynomial
of general graphs and used it to classify several graph
families~\cite{mojdeh}; and Yoosuf and Kuttipulackal computed it for several
graph constructions~\cite{yoosuf}. More recently, Caputi, Celoria and Collari
studied connected domination from a homological viewpoint, proving that the Euler
characteristic of a certain graph homology equals the connected domination
polynomial evaluated at a particular point~\cite{ccc}. (For the related total and
independent domination polynomials we refer to~\cite{cdh,allan,jafari} and the
references therein.) None of this work, however, addresses the location of the
connected domination roots; that is the question studied here.

\subsection{The substitution formula of Brown and Tufts}

The following substitution formula of Brown and Tufts underlies all of our main
results and will be used throughout.

\begin{lemma}[Brown and Tufts, {\cite{bt}}]\label{lem:bt}
Let $G$ be an arbitrary graph and let $K_n$ be the complete graph of order $n$.
Then
\begin{equation}
D(G[K_n],x)=D\big(G,(x+1)^n-1\big).
\end{equation}
\end{lemma}

This formula provides a mechanism for spreading roots: if $z_1$ is a
domination root, then all $n$ solutions of the equation $(z+1)^n-1=z_1$ are again
domination roots (of the graph $G[K_n]$). Brown and Tufts used this mechanism
to prove that domination roots are dense in the complex
plane~\cite{bt}, and it is also the key tool in the proof by Brown and Beaton
that the closure of the real domination roots is $(-\infty,0]$~\cite{bb}. To the
best of our knowledge, no analogous substitution formula has been recorded for
the connected domination polynomial.

\subsection{Contribution and organization}

If the connected domination polynomial obeyed a similar substitution formula,
the same mechanism would become available for connected domination roots.
Section~\ref{sec:sub} proves such a formula (Theorem~\ref{thm:cd}); the remaining
sections use it to determine the distribution of the roots. Section~\ref{sec:real}
proves that the closure of the real connected domination roots is $(-\infty,0]$
(Theorem~\ref{thm:real}), taking the real roots of the cycles $C_n$ as seeds for
the part $(-\infty,-1)$ and a new construction (the real roots of the joins
$C_m\vee C_n$ lying in $(-1,0)$) as seeds for the remaining interval; the
join construction can already be seen numerically: $C_5\vee C_5$ has a root
$\approx-0.587$ in $(-1,0)$, and the roots of $C_m\vee C_n$ approach $-1$ from
the right as $m,n$ grow (Example~\ref{ex:joinseed} and
Figure~\ref{fig:seeds}).
Section~\ref{sec:complex} proves that the closure of the connected domination
roots is the whole complex plane (Theorem~\ref{thm:complex}). The substitution
behaviour of the independent and total domination polynomials is compared in
Table~\ref{tab:sub}; neither of them provides a spreading mechanism, so the
density method developed here is specific to connected domination. Concluding
remarks and open problems appear in Section~\ref{sec:concl}.

We close this section by isolating what is genuinely new here. Although the
overall strategy, a substitution formula combined with explicit seed
families, parallels the framework of~\cite{bt,bb}, two ingredients of the
present paper have no counterpart in the ordinary domination setting. First,
the proof of Theorem~\ref{thm:cd} rests on a connectivity-preserving
correspondence between the connected dominating sets of $G[K_n]$ and those of
$G$: for the ordinary domination polynomial only the domination equivalence is
needed, whereas here one must additionally show that connectedness is preserved
under projection and under lifting (Section~\ref{sec:sub}), an analysis that is
specific to connected domination. Second, to the best of our knowledge no
family of connected domination roots approaching $-1$ from the right was
previously available---the cycle roots of Lemma~\ref{lem:cycle} approach $-1$
only from the left---so the interval $(-1,0)$ could not be reached by known
seeds; the construction of right-hand seeds via the joins $C_m\vee C_n$
(Lemma~\ref{lem:joinseed}) is new and is indispensable for the second half of
Theorem~\ref{thm:real}. Both ingredients are used in what follows, and neither
has an analogue in the ordinary domination theory.

\section{Substitution formula for the connected domination polynomial}\label{sec:sub}

In this section we establish the main tool of the paper: the connected domination
polynomial satisfies the same substitution formula as the ordinary domination
polynomial (Lemma~\ref{lem:bt}). We keep the notation
$\pi(S)=\{v:S\cap V(K_n^v)\neq\emptyset\}$ and repeatedly use the following two
structural facts about the substitution $G[K_n]$.

\paragraph{Fact A (domination and connectivity inside a clique).} Any vertex of
the clique $K_n^v$ dominates the whole clique; every nonempty subset of a clique
induces a connected subgraph.

\paragraph{Fact B (complete bipartite connection between cliques).} If
$u\sim v$, then every vertex of $K_n^u$ is adjacent to every vertex of $K_n^v$.

\begin{theorem}\label{thm:cd}
Let $G$ be a connected graph. Then
\begin{equation}
D_c(G[K_n],x)=D_c\big(G,(x+1)^n-1\big).
\end{equation}
\end{theorem}

\begin{proof}
Let $S\subseteq V(G[K_n])$, write $T=\pi(S)\subseteq V(G)$, and put
$|V(G)|=m$, so that $|V(G[K_n])|=mn$, where $\pi$ is the projection defined
above. We claim that
\[
S\text{ is a connected dominating set of }G[K_n]
\iff T=\pi(S)\text{ is a connected dominating set of }G.
\]

\emph{Equivalence of domination.}
\begin{itemize}
\item[($\Rightarrow$)] Suppose that $S$ dominates $G[K_n]$, and take any
$v\in V(G)$. If $v\in T$, then $K_n^v$ is dominated by $S\cap K_n^v$; in
particular $v$ itself is dominated by $T$ in $G$ (since $v\in N[T]$). If
$v\notin T$, then $K_n^v$ contains no vertex of $S$; any vertex $(v,j)$ of
$K_n^v$ is dominated by some $s=(u,i)\in S$, and from
$(u,i)(v,j)\in E(G[K_n])$ with $u\neq v$ (otherwise $v\in T$) we obtain
$u\sim v$. Hence $u\in T$ and $u\sim v$, that is, $v$ is dominated by $T$ in
$G$.
\item[($\Leftarrow$)] Suppose that $T$ dominates $G$, and take any
$(v,j)\in V(G[K_n])$. If $v\in T$, then $K_n^v\cap S\neq\emptyset$, and
$(v,j)$ has a neighbour of $S$ inside $K_n^v$; if $v\notin T$, there exists
$u\in T$ with $u\sim v$, and by Fact~B every vertex of $K_n^u$ is adjacent to
every vertex of $K_n^v$, so $(v,j)$ is dominated by $S\cap K_n^u$.
\end{itemize}

\emph{Equivalence of connectivity.}
\begin{itemize}
\item[($\Rightarrow$)] Suppose that $G[S]$ is connected; we show that $G[T]$
is connected. By the definition of the edge set of $G[K_n]$,
$(u,i)\sim(v,j)$ implies $u=v$ or $u\sim v$; hence $\pi$ maps every edge of
$G[S]$ to an edge of $G$ (or collapses it to a single vertex). Consequently
$\pi$ maps every path of $G[S]$ to a walk in $G[T]$: if
$s_0s_1\cdots s_\ell$ is a path in $G[S]$, then
$\pi(s_0)\pi(s_1)\cdots\pi(s_\ell)$ is a walk in $G[T]$ in which every two
consecutive terms are adjacent or equal, and after deleting repetitions one
obtains a path in $G[T]$ joining $\pi(s_0)$ to $\pi(s_\ell)$. Since $G[S]$
is connected, any two vertices $u,v\in T$ are joined by the projection of a
path of $G[S]$, so $G[T]$ is connected.

\item[($\Leftarrow$)] Suppose that $G[T]$ is connected, and write
$S=\bigcup_{v\in T}A_v$, where $A_v:=S\cap V(K_n^v)\neq\emptyset$ for
$v\in T$. We prove that $G[S]$ is connected, in two observations:
  \begin{enumerate}
  \item[(i)] Each $A_v$ is connected in $G[S]$: by Fact~A, $K_n^v$ is a
  clique (hence connected), so the subgraph induced by
  $A_v\subseteq V(K_n^v)$ is connected.
  \item[(ii)] If $uv\in E(G[T])$ (that is, $u,v\in T$ and $u\sim v$), then
  $A_u\cup A_v$ is connected in $G[S]$: by Fact~B every vertex of $K_n^u$ is
  adjacent to every vertex of $K_n^v$, so every vertex of $A_u$ is
  adjacent to every vertex of $A_v$; thus $G[A_u\cup A_v]$ contains a
  complete bipartite graph between $A_u$ and $A_v$ (together with the clique
  edges inside $K_n^u$ and $K_n^v$) and is therefore connected.
  \end{enumerate}
Take a spanning tree $T'$ of $G[T]$ (which exists since $G[T]$ is
connected). By~(ii), the union of the blocks $A_u,A_v$ along every edge of
$T'$ is connected, and by~(i) each block $A_v$ is internally connected.
Starting from the block containing the root of $T'$ and merging the blocks
one by one along the tree edges, the union remains connected after each
merging step; hence $G[S]$ is connected.
\end{itemize}

Let $T$ be a fixed connected dominating set of $G$ of size $k$ ($k\ge1$, a
connected dominating set being nonempty). By the claim above, $\pi(S)=T$ if and
only if
\[
S\cap K_n^v\neq\emptyset\;\;(v\in T),\qquad
S\cap K_n^u=\emptyset\;\;(u\notin T).
\]
Inside the clique $K_n^v$ with $v\in T$ an arbitrary nonempty subset may be
chosen, contributing the weight $\sum_{j=1}^{n}\binom{n}{j}x^j$ (the term $x^j$
recording a subset of size $j$); inside the cliques with $u\notin T$ the empty
set is chosen, contributing $1$. Hence, for fixed $T$, the total weight of all
$S$ with $\pi(S)=T$ (which are connected dominating sets of $G[K_n]$) is
\[
\left(\sum_{j=1}^{n}\binom{n}{j}x^j\right)^k\cdot 1^{m-k}
=\big((x+1)^n-1\big)^k.
\]

Finally, every $S\subseteq V(G[K_n])$ determines a unique $T=\pi(S)$. Moreover,
$S$ is a connected dominating set of $G[K_n]$ if and only if $T$ is a connected
dominating set of $G$; and for a fixed connected dominating set $T$ of $G$ the
total weight of all $S$ with $\pi(S)=T$ is $\big((x+1)^n-1\big)^{|T|}$.
Therefore
\[
\begin{aligned}
D_c(G[K_n],x)&=\sum_{T}\big((x+1)^n-1\big)^{|T|}\\
&=D_c\big(G,(x+1)^n-1\big),
\end{aligned}
\]
where the sum runs over all connected dominating sets $T$ of $G$, and the last
step is the same polynomial composition as in Lemma~\ref{lem:bt}:
$D_c(G,y)=\sum_k d_c(G,k)y^k$, and substituting $y=(x+1)^n-1$ gives the
claim.\qedhere
\end{proof}

\begin{example}[Numerical verification]\label{ex:cd}
Take $G=P_3$ (the path $v_1\!-\!v_2\!-\!v_3$) and $n=2$. The connected
dominating sets of $P_3$ are: $\{v_2\}$ (size $1$); $\{v_1,v_2\}$,
$\{v_2,v_3\}$ (note that $v_1$ and $v_3$ are not adjacent, so $\{v_1,v_3\}$
induces a disconnected subgraph of $P_3$ and is excluded); and
$\{v_1,v_2,v_3\}$ (size $3$). Hence $D_c(P_3,x)=x+2x^2+x^3$.

By Theorem~\ref{thm:cd},
\[
D_c(P_3[K_2],x)=D_c(P_3,(x+1)^2-1)=y+2y^2+y^3,\quad
y=(x+1)^2-1=x^2+2x.
\]
Expanding: $y=x^2+2x$, $y^2=x^4+4x^3+4x^2$, $y^3=x^6+6x^5+12x^4+8x^3$, so
\[
D_c(P_3[K_2],x)=2x+9x^2+16x^3+14x^4+6x^5+x^6.
\]
Enumerating all $2^6=64$ subsets of $P_3[K_2]$ (on $6$ vertices) and checking
each for connected domination confirms the formula:
$d_c(P_3[K_2],1)=2$ (the two vertices of the middle clique),
$d_c(P_3[K_2],2)=9$, $d_c(P_3[K_2],3)=16$, $d_c(P_3[K_2],4)=14$,
$d_c(P_3[K_2],5)=6$ and $d_c(P_3[K_2],6)=1$.
\end{example}

\begin{remark}\label{rem:whyKn}
The restriction to complete factors $H=K_n$ in Theorem~\ref{thm:cd} is
essential, and it explains why Table~\ref{tab:sub} below concerns only the
substitution $G[K_n]$. The completeness of the copies is used in two places in
the proof: a vertex of $K_n^v\setminus S$ is dominated by $S\cap K_n^v$ only
because $K_n^v$ is a clique (Fact~A), and the union $\bigcup_{v\in T}A_v$ is
connected only because every nonempty $A_v$ induces a connected subgraph
(Fact~A again). Both properties fail for non-complete $H$: for instance, if
$G=K_2$ and $H=\overline{K_2}$, then $G[H]=K_{2,2}$, and the set $S$ formed by
the two vertices of a single copy $H^u$ is a dominating set of $K_{2,2}$ that
induces a disconnected subgraph, although $\pi(S)$ (a single vertex) is a
connected dominating set of $K_2$; the correspondence of Theorem~\ref{thm:cd}
thus breaks down. No substitution formula for $D_c(G[H],x)$ with general $H$
is known.
\end{remark}

For the other two classical variants, independent and total domination, the
substitution question has already been settled in the
literature; for ease of comparison we restate the corresponding results below and
make no claim to their novelty.

\paragraph{Independent domination.} Jahari and Alikhani~\cite{jahari} gave the
closed form
\begin{equation}
D_i(G[H],x)=D_i\big(G,D_i(H,x)\big)
\end{equation}
for the independent domination polynomial of an arbitrary lexicographic product;
its specialisation to $H=K_n$ (where $D_i(K_n,x)=nx$) reads
\begin{equation}
D_i(G[K_n],x)=D_i(G,nx).
\end{equation}
Here an independent set takes at most one vertex in each
clique, so the freedom inside a clique degenerates from ``an arbitrary nonempty
subset'' to ``a single vertex''. Unlike Theorem~\ref{thm:cd}, the substitution
here is a linear rescaling of the variable, and therefore provides no
spreading mechanism of the type of Lemma~\ref{lem:bt}.

\paragraph{Total domination.} Dod~\cite{dod} introduced a trivariate total
domination polynomial and used it to determine the total domination polynomial of
the lexicographic product $G[K_n]$. For total domination the Brown--Tufts form
fails: two vertices inside a clique can ``support each other'', so that an entire
total dominating set collapses inside a single clique; after projection the
vertex of $G$ corresponding to that clique becomes isolated in $\pi(S)$, and the
bijection between $S$ and $\pi(S)$ breaks down. A substitution formula for total
domination must therefore keep track of the number of isolated vertices of the
projected set, a complication that does not arise for ordinary or connected
domination.

The three variants are compared in Table~\ref{tab:sub}. Only connected domination
retains the Brown--Tufts form of Theorem~\ref{thm:cd}; the density theorems of
Sections~\ref{sec:real} and~\ref{sec:complex} rely on this
mechanism. For the independent and total domination polynomials no
spreading mechanism is available, and the distribution of their roots remains
unknown (see Section~\ref{sec:concl}).

The structure of connected and total dominating sets in lexicographic
products has also been investigated from the viewpoint of forcing subsets by
Armada, Canoy and Go~\cite{armada}; their results concern the domination
parameters rather than the counting polynomials studied here.

\begin{table}[htbp]
\centering
\caption{Substitution behaviour of the classical domination variants under $G[K_n]$}
\label{tab:sub}
\begin{tabular}{@{}>{\raggedright\arraybackslash}p{2.0cm}>{\raggedright\arraybackslash}p{5.0cm}>{\raggedright\arraybackslash}p{1.8cm}>{\raggedright\arraybackslash}p{1.6cm}>{\raggedright\arraybackslash}p{3.2cm}@{}}
\toprule
Parameter & Formula for $G[K_n]$ & Type & Spread & Roots \\
\midrule
$D$ & $D(G[K_n],x)=D(G,(x+1)^n-1)$~\cite{bt} & power & yes & closure
$(-\infty,0]$; dense in $\C$~\cite{bt,bb} \\
$D_c$ & $D_c(G[K_n],x)=D_c(G,(x+1)^n-1)$ & power & yes & Theorem~\ref{thm:real},
\ref{thm:complex} (this paper) \\
$D_i$ & $D_i(G[K_n],x)=D_i(G,nx)$~\cite{jahari} & scaling & no & unknown \\
$D_t$ & trivariate formula of~\cite{dod} & complex & no & unknown \\
\bottomrule
\end{tabular}
\end{table}

\section{The closure of real connected domination roots is $(-\infty,0]$}\label{sec:real}

In this section we prove that the closure of the real connected domination roots
is $(-\infty,0]$. The proof rests on two families of seed roots: the real roots
of the connected domination polynomials of the cycles $C_n$, and the real roots
of the joins $C_m\vee C_n$ (see Definition~\ref{def:join}) lying in $(-1,0)$.
The former supply seeds for the part $(-\infty,-1)$ and the latter for the part
$(-1,0)$; the backward closure under graph substitution then spreads them over
the whole interval.

\subsection{The connected domination polynomial of cycles}

The closed form in the following lemma was obtained by Dhananjaya Murthy et
al.~\cite[Theorem~2.4]{murthy}. We include the short proof for completeness
and then derive the asymptotic behaviour of the roots needed later.

\begin{lemma}\label{lem:cycle}
For $n\ge4$,
\begin{equation}
D_c(C_n,x)=x^{n-2}\big(x^2+nx+n\big),
\end{equation}
whose real roots are $0$ (with multiplicity $n-2$) together with
\begin{equation}
z_n^{\pm}=\frac{-n\pm\sqrt{n^2-4n}}{2}\qquad
(\,n\ge5:\ z_n^+\in(-2,-1),\ z_n^-<-2;\quad n=4:\ z_4^+=z_4^-=-2\,).
\end{equation}
As $n\to\infty$, $z_n^+\to-1^-$ and $z_n^-=-n+1+o(1)\to-\infty$.
\end{lemma}

\begin{proof}
Let $S$ be a connected dominating set of the cycle $C_n$. The connected proper
subsets of a cycle are its arcs, so the complement $\overline{S}$ is also
an arc. Now $S$ dominates $C_n$ if and only if every vertex of
$\overline{S}$ has a neighbour in $S$, which is equivalent to the arc
$\overline{S}$ having at most $2$ vertices (if it has at least $3$ vertices,
the middle vertex of the arc has both neighbours in $\overline{S}$). Hence
$S$ is either an arc of $n-2$ vertices ($n$ positions), an arc of $n-1$
vertices ($n$ positions), or the whole vertex set ($1$ possibility), so that
\begin{equation}
D_c(C_n,x)=n\,x^{n-2}+n\,x^{n-1}+x^n=x^{n-2}(x^2+nx+n).
\end{equation}
The discriminant of the quadratic factor is $n^2-4n\ge0$ for $n\ge4$, and its
two roots are $z_n^\pm$ (for $n=4$ the discriminant vanishes and
$z_4^+=z_4^-=-2$). For $z_n^+$, rationalising gives
\begin{equation}
z_n^++1=\frac{\sqrt{n^2-4n}-(n-2)}{2}
=\frac{-2}{\sqrt{n^2-4n}+n-2}<0,
\end{equation}
\begin{equation}
|z_n^++1|=\frac{2}{\sqrt{n^2-4n}+n-2}\approx\frac1n\to0,
\end{equation}
so that $z_n^+$ tends to $-1$ from the left, with $z_n^+\in(-2,-1)$ for
$n\ge5$. For $z_n^-$, the expansion $\sqrt{n^2-4n}=n-2+o(1)$ gives
$z_n^-=-n+1+o(1)\to-\infty$, and for each $n\ge4$ there is exactly one such root.
\end{proof}

Denote by
$\mathcal{R}_c=\{z\in\C:\text{there is a connected graph }G\text{ with }
D_c(G,z)=0\}$ the set of all connected domination roots.

\begin{lemma}[Backward closure]\label{lem:backward}
If $z_1\in\mathcal{R}_c$, then for every positive integer $m$, all solutions
of the equation $(w+1)^m-1=z_1$ belong to $\mathcal{R}_c$ (they are distinct
unless $z_1=-1$, in which case $w=-1$ is the unique solution).
\end{lemma}

\begin{proof}
Suppose $D_c(G,z_1)=0$ with $G$ connected. By Theorem~\ref{thm:cd},
$D_c(G[K_m],x)=D_c(G,(x+1)^m-1)$. Evaluating the polynomial $D_c(G,y)$ at
$y=z_1$, the equation $(w+1)^m-1=z_1$ says that $x=w$ is a root of
$D_c(G,(x+1)^m-1)=0$, hence $D_c(G[K_m],w)=0$. Moreover, $G[K_m]$ is
connected whenever $G$ is: given any two vertices of $G[K_m]$, connectedness
of $G$ provides, for their first coordinates $u,v$, a path
$u=v_0v_1\cdots v_\ell=v$ in $G$, and
\[
(u,1)=(v_0,1)\sim(v_1,1)\sim\cdots\sim(v_\ell,1)=(v,1)
\]
is a walk in $G[K_m]$ joining the two fibres, each step using
$(v_i,1)(v_{i+1},1)\in E(G[K_m])$ (since $v_i\sim v_{i+1}$); within each
fibre the clique edges complete the connection. Hence $G[K_m]$ is connected
and $w\in\mathcal{R}_c$.\qedhere
\end{proof}

\subsection{Right-hand seeds: joins of cycles}

\begin{definition}[Join]\label{def:join}
The \emph{join} $G_1\vee G_2$ of two graphs $G_1$ and $G_2$ is obtained from
their disjoint union by joining every vertex of $V(G_1)$ to every vertex of
$V(G_2)$.
\end{definition}

The following formula for the join is due to Dhananjaya Murthy et
al.~\cite[Theorem~2.7]{murthy}. The seed construction of
Lemma~\ref{lem:joinseed} builds directly on it.

\begin{lemma}[Join formula, Murthy et al.~{\cite{murthy}}]\label{lem:join}
Let $G_i$ be a graph of order $n_i$ ($i=1,2$). Then
\begin{equation}
D_c(G_1\vee G_2,x)=D_c(G_1,x)+D_c(G_2,x)
+\big((x+1)^{n_1}-1\big)\big((x+1)^{n_2}-1\big).
\end{equation}
\end{lemma}

\begin{lemma}[Right-hand seeds]\label{lem:joinseed}
For every $\delta\in(0,1)$ there exist sufficiently large odd integers $m,n$
such that $D_c(C_m\vee C_n,x)$ has a real root in $(-1,-1+\delta)$: real
connected domination roots approach $-1$ arbitrarily closely from
the right.
\end{lemma}

\begin{proof}
Fix $\delta\in(0,1)$; in what follows $m,n$ are odd. By Lemma~\ref{lem:cycle},
\begin{equation}
D_c(C_k,-1)=(-1)^{k-2}\big(1-k+k\big)=(-1)^{k-2},
\end{equation}
so $D_c(C_k,-1)=-1$ for odd $k$. Since $(1+x)^k\big|_{x=-1}=0$,
Lemma~\ref{lem:join} gives
\begin{equation}
D_c(C_m\vee C_n,-1)=-1+(-1)+(0-1)(0-1)=-1<0.
\end{equation}
On the other hand, at $x=-1+\delta$,
\begin{equation}
D_c(C_k,-1+\delta)=(-1+\delta)^{k-2}\big(1+(k-2)\delta+\delta^2\big),
\end{equation}
whose modulus is at most $(1-\delta)^{k-2}(1+k\delta+\delta^2)\to0$ as
$k\to\infty$; and
\begin{equation}
\big((1+x)^m-1\big)\big((1+x)^n-1\big)\Big|_{x=-1+\delta}
=(\delta^m-1)(\delta^n-1)\to1\qquad(m,n\to\infty).
\end{equation}
Hence $D_c(C_m\vee C_n,-1+\delta)>0$ for sufficiently large odd $m,n$, and the
intermediate value theorem yields a real root of $D_c(C_m\vee C_n,x)$ in
$(-1,-1+\delta)$.
\end{proof}

\begin{example}[Numerical verification]\label{ex:joinseed}
$D_c(C_5\vee C_5,x)$ has a root $\approx-0.587$ in $(-1,0)$;
$D_c(C_5\vee C_{19},x)$ has a root $\approx-0.818$; and
$D_c(C_{19}\vee C_{19},x)$ has a root $\approx-0.908$, approaching $-1$ as
$m,n$ grow. These and further values are displayed in
Figure~\ref{fig:seeds}.
\end{example}

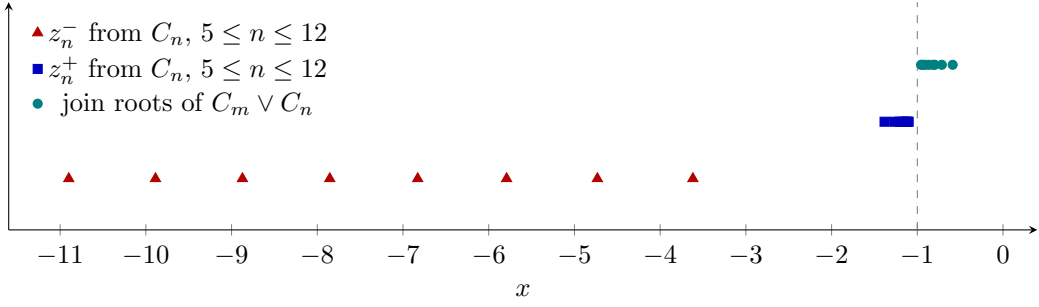
\begin{figure}[htbp]
\centering
\begin{tikzpicture}
\begin{axis}[
    width=0.92\textwidth, height=4.6cm,
    xmin=-11.6, xmax=0.4, ymin=-0.45, ymax=1.55,
    axis lines=left,
    xlabel={$x$},
    ytick=\empty,
    xtick={-11,-10,-9,-8,-7,-6,-5,-4,-3,-2,-1,0},
    x tick label style={font=\small},
    xlabel style={font=\small},
    legend style={font=\small,draw=none,at={(0.015,0.97)},anchor=north west},
]
\addplot[only marks,mark=triangle*,mark size=2.2pt,color=red!70!black] coordinates
{(-3.618034,0) (-4.732051,0) (-5.791288,0) (-6.828427,0) (-7.854102,0)
(-8.872983,0) (-9.887482,0) (-10.898979,0)};
\addlegendentry{$z_n^-$ from $C_n$, $5\le n\le12$}
\addplot[only marks,mark=square*,mark size=1.7pt,color=blue!75!black] coordinates
{(-1.381966,0.5) (-1.267949,0.5) (-1.208712,0.5) (-1.171573,0.5)
(-1.145898,0.5) (-1.127017,0.5) (-1.112518,0.5) (-1.101021,0.5)};
\addlegendentry{$z_n^+$ from $C_n$, $5\le n\le12$}
\addplot[only marks,mark=*,mark size=1.7pt,color=teal] coordinates
{(-0.587440,1.0) (-0.714073,1.0) (-0.795980,1.0) (-0.817949,1.0)
(-0.865988,1.0) (-0.908222,1.0) (-0.937868,1.0) (-0.956195,1.0)};
\addlegendentry{join roots of $C_m\vee C_n$}
\draw[dashed,gray] (axis cs:-1,-0.45) -- (axis cs:-1,1.55);
\end{axis}
\end{tikzpicture}
\caption{The two seed families on the real line (the three horizontal lanes
are for display only). Triangles: the cycle roots $z_n^-=-n+1+o(1)$,
diverging to $-\infty$; squares: the cycle roots $z_n^+\to-1^{-}$; circles:
real roots of the joins $C_m\vee C_n$ for the odd pairs
$(m,n)=(5,5),(5,9),(9,9),(5,19),(9,19),(19,19),(19,39),(39,39)$, approaching
$-1$ from the right (the first, fourth and sixth values are those of
Example~\ref{ex:joinseed}).}
\label{fig:seeds}
\end{figure}

\subsection{Main theorem}

\begin{theorem}\label{thm:real}
The closure of the real connected domination roots is $(-\infty,0]$.
\end{theorem}

\begin{proof}
First, all coefficients of $D_c(G,x)$ are nonnegative, so there is no positive
real root, and $0$ is always a root. It remains to prove density. Let
$z\in(-\infty,0)$ and $\varepsilon>0$ be given (taken small enough that the
intervals involved avoid the endpoints); we construct a real connected domination
root in $(z-\varepsilon,z+\varepsilon)$. Write $f_m(w)=(w+1)^m-1$, with $m$
odd, so that $f_m:\R\to\R$ is a strictly increasing bijection.

\emph{Case 1: $z\in(-2,-1)$.} Take
$0<\varepsilon<\min\{|z+1|,\,1-|z+1|\}$ and put
\[
a=|z+1|+\varepsilon,\qquad b=|z+1|-\varepsilon.
\]
Then $0<b<a<1$. Since $m$ is odd,
\[
f_m\big((z-\varepsilon,z+\varepsilon)\big)=\big(-1-a^m,\,-1-b^m\big).
\]
On the other hand, by Lemma~\ref{lem:cycle},
\[
\delta_n:=|-1-z_n^+|=\frac{2}{\sqrt{n^2-4n}+n-2},
\]
and for $n\ge5$ we have
\[
\frac{1}{2n}<\delta_n<\frac{2}{n}.
\]
Now
\[
\frac{1/(2b^m)}{2/a^m}=\frac14\Big(\frac{a}{b}\Big)^m\longrightarrow\infty,
\]
so for all sufficiently large odd $m$ there exists an integer $n\ge5$ with
\[
\frac{2}{a^m}<n<\frac{1}{2b^m}.
\]
Consequently
\[
b^m<\frac{1}{2n}<\delta_n<\frac{2}{n}<a^m,
\]
so that
\[
z_n^+=-1-\delta_n\in(-1-a^m,-1-b^m)
=f_m\big((z-\varepsilon,z+\varepsilon)\big).
\]
By Lemma~\ref{lem:backward}, the real solution $w=f_m^{-1}(z_n^+)$ of the
equation $f_m(w)=z_n^+$ is a connected domination root; and since $f_m$ is
strictly increasing on $\R$, we have $w\in(z-\varepsilon,z+\varepsilon)$.

\emph{Case 2: $z\in(-1,0)$.} The left-hand seeds $z_n^+$ always lie to the left
of $-1$ and cannot reach an interval on the right, so we use the right-hand
seeds provided by Lemma~\ref{lem:joinseed} instead. Since $z+1\in(0,1)$, set
\[
\delta=\min\big\{z+1,\ e^{-4e^{-2}/\varepsilon},\ \tfrac12\big\}\in(0,1).
\]
By Lemma~\ref{lem:joinseed}, applied with this $\delta$, there exist
sufficiently large odd integers $m',n'$ and a real connected domination root
$\rho\in(-1,-1+\delta)\subseteq(-1,z)$ of $D_c(C_{m'}\vee C_{n'},x)$; then
$\rho+1<\delta\le e^{-4e^{-2}/\varepsilon}$, so writing $\rho+1=e^{-\eta}$
with $\eta=-\ln(\rho+1)>0$ we have $\eta>4e^{-2}/\varepsilon$.

Say $D_c(H,\rho)=0$, where $H$ is the join $C_{m'}\vee C_{n'}$ constructed in
Lemma~\ref{lem:joinseed}. For each positive integer $r$, let
\[
w_r=(\rho+1)^{1/r}-1.
\]
Then $(w_r+1)^r-1=(\rho+1)-1=\rho$, so by Lemma~\ref{lem:backward}, $w_r$ is a
real connected domination root of $H[K_r]$ (here $r$ is an arbitrary positive
integer, not necessarily odd: when $r$ is even, $(w_r+1)^r$ is a positive
number in $(0,1)$ for $w_r\in(-1,0)$, and since $\rho\in(-1,0)$ we have
$\rho+1\in(0,1)$, so every positive integer $r$ works).

The sequence $\{w_r\}$ has the following properties:
\begin{itemize}
\item $w_r\in(-1,0)$ for every positive integer $r$: indeed
$w_r+1=(\rho+1)^{1/r}\in(0,1)$;
\item $w_r$ is strictly increasing: $(\rho+1)^{1/r}$ is increasing in $r$
(for $\rho+1<1$, the smaller $1/r$ is, the larger $(\rho+1)^{1/r}$ is, so
larger $r$ gives larger $w_r$);
\item $w_1=\rho<z$ (since $\rho\in(-1,z)$);
\item $\lim_{r\to\infty}w_r=0>z$ (since $(\rho+1)^{1/r}\to1$ for
$\rho+1\in(0,1)$);
\item the consecutive gaps $w_{r+1}-w_r$ are uniformly smaller than
$\varepsilon$: by the mean value theorem applied to the function
$g(t)=(\rho+1)^{1/t}-1=e^{-\eta/t}-1$ (with $\eta=-\ln(\rho+1)$), whose
derivative is $g'(t)=\eta t^{-2}e^{-\eta/t}$, there exists $\xi\in(r,r+1)$
such that $g(r+1)-g(r)=g'(\xi)=\eta\xi^{-2}e^{-\eta/\xi}$.
\end{itemize}

We estimate $g'(\xi)$ further. Since $\xi\in(r,r+1)$, putting
$u=\eta/\xi\in(\eta/(r+1),\eta/r)$ we get
\[
g'(\xi)=\frac{\eta}{\xi^2}e^{-\eta/\xi}
=\frac{1}{\eta}\left(\frac{\eta}{\xi}\right)^2e^{-\eta/\xi}
=\frac{1}{\eta}u^2e^{-u}.
\]
The function $h(u)=u^2e^{-u}$ on $u>0$ attains its maximum where
$h'(u)=2ue^{-u}-u^2e^{-u}=ue^{-u}(2-u)=0$, namely at $u=2$; and
$h(2)=4e^{-2}$. Hence $h(u)\le4e^{-2}$ for all $u>0$, and therefore
\[
g'(\xi)=\frac{1}{\eta}h(u)\le\frac{4e^{-2}}{\eta}<\varepsilon
\]
(the last step using the choice $\eta>4e^{-2}/\varepsilon$). Thus
$|w_{r+1}-w_r|=|g(r+1)-g(r)|<\varepsilon$ for all $r$.

The sequence $\{w_r\}$ increases strictly from $w_1=\rho<z$ to
$\lim_{r\to\infty}w_r=0>z$, passing from below $z$ to above $z$ while
consecutive terms differ by less than $\varepsilon$; hence some $w_r$ lies in
$(z-\varepsilon,z+\varepsilon)$.

\emph{Case 3: $z\in(-\infty,-2)$.} Take
$0<\varepsilon<-z-2$ and put
\[
c=|z+1|>1.
\]
Then for odd $m$,
\[
f_m\big((z-\varepsilon,z+\varepsilon)\big)
=I_m:=\big(-(c+\varepsilon)^m-1,\,-(c-\varepsilon)^m-1\big).
\]
Write
\[
A_m=(c+\varepsilon)^m,\qquad B_m=(c-\varepsilon)^m.
\]
Since $c-\varepsilon>1$, we have
\[
A_m-B_m\longrightarrow\infty.
\]
By Lemma~\ref{lem:cycle} we have $z_n^-=-n+1+o(1)$; choose $N_0$ such that
\[
|z_n^- - (-n+1)|<\tfrac14
\]
for all $n\ge N_0$. Next choose an odd integer $m$ large enough that
\[
A_m-B_m>3\quad\text{and}\quad B_m+\tfrac52\ge N_0
\]
(the latter condition being ensured by $B_m=(c-\varepsilon)^m\to\infty$). Then
the interval
\[
\left(B_m+\frac52,\ A_m+\frac32\right)
\]
has length greater than $1$, so it contains an integer $n\ge N_0$. For this $n$,
\[
B_m+\frac52<n<A_m+\frac32.
\]
Hence
\[
-A_m-1+\frac12<-n+1<-B_m-1-\frac12.
\]
Combining this with $|z_n^- - (-n+1)|<\tfrac14$ we obtain
\[
-A_m-1<z_n^-<-B_m-1,
\]
that is, $z_n^-\in I_m=f_m\big((z-\varepsilon,z+\varepsilon)\big)$. By
Lemma~\ref{lem:backward}, the real solution $w$ of the equation $f_m(w)=z_n^-$
is a connected domination root; and since $f_m$ is strictly increasing on $\R$,
we have $w\in(z-\varepsilon,z+\varepsilon)$.

\emph{Endpoints.} The numbers $z=-2$, $z=-1$ and $z=0$ are themselves connected
domination roots ($D_c(K_2,x)=x(x+2)$, $D_c(P_3,x)=x(1+x)^2$, and $0$ is
always a root), so they belong to the closure.

Combining these cases, $(-\infty,0]\subseteq\overline{\mathcal{R}_c\cap\R}$,
while the reverse inclusion follows from the nonnegativity of the coefficients.
\end{proof}

\section{The closure of connected domination roots is the whole complex plane}\label{sec:complex}

Below we prove that the connected domination roots are dense in the whole complex
plane; recall that $\mathcal{R}_c$ denotes the set of all connected domination
roots.

\begin{lemma}\label{lem:circle}
$\mathcal{R}_c$ is dense on the circle $\{z:|z+1|=1\}$.
\end{lemma}

\begin{proof}
Let $G$ be any connected graph. Every connected domination polynomial has zero
constant term (the empty set is not a dominating set), so $0$ is a root of
$D_c(G,y)$. By Theorem~\ref{thm:cd},
$D_c(G[K_n],x)=D_c(G,(x+1)^n-1)$. Since $D_c(G,0)=0$, all $n$ solutions of
the equation $(x+1)^n-1=0$ are roots of $D_c(G[K_n],x)$, namely
\[
x=-1+e^{2\pi i j/n},\quad j=0,1,\dots,n-1.
\]
These points lie on the circle $|x+1|=1$ (since $|e^{2\pi i j/n}|=1$).

By the density of the roots of unity: the points $-1+e^{2\pi i j/n}$,
$j=0,1,\dots,n-1$, are equally spaced on the circle $|z+1|=1$ with angular
step $2\pi/n$; as $n\to\infty$ they become dense on this circle (that is,
for every point $p$ on the circle and every $\delta>0$, there exist $n$ and
$j$ with $|-1+e^{2\pi i j/n}-p|<\delta$). Hence $\mathcal{R}_c$ is dense on
the circle $|z+1|=1$.\qedhere
\end{proof}

\begin{theorem}[Main theorem]\label{thm:complex}
The closure of the connected domination roots is the whole complex plane:
$\overline{\mathcal{R}_c}=\C$.
\end{theorem}

\begin{proof}
Let $w_0\in\C$ and $\varepsilon>0$ be given. It suffices to produce
$w\in\mathcal{R}_c$ with $|w-w_0|<\varepsilon$.

\emph{Case 1: $w_0=-1$.} Note that $-1\in\mathcal{R}_c$ itself, since
$D_c(P_3,x)=x(1+x)^2$ has $-1$ as a root (of multiplicity $2$). Then
$|w-w_0|=0<\varepsilon$, and the conclusion is trivial.

\emph{Case 2: $w_0\neq-1$.} Write $r=|w_0+1|>0$ and
$\theta=\arg(w_0+1)\in[0,2\pi)$. Choose an odd integer $m$ large enough that
some branch angle $\theta_j=\frac{(2j+1)\pi}{m}$
($j\in\{0,\dots,m-1\}$) satisfies
\[
|\theta_j-\theta|<\varepsilon/(2r).
\]
Such a $j$ exists: for odd $m$ large enough, the angles
$\theta_j=\pi/m,3\pi/m,\dots,(2m-1)\pi/m$ are equally spaced around the
circle with step $2\pi/m$, so there is $j$ with
$|\theta_j-\theta|\le\pi/m<\varepsilon/(2r)$ (taking $m$ large enough; here
$|\theta_j-\theta|$ denotes the circular distance, the difference being
taken modulo $2\pi$).

Let $w^*=-1+r\,e^{i\theta_j}$. Then
\begin{align}
|w^*-w_0|&=|-1+r\,e^{i\theta_j}-(-1+r\,e^{i\theta})|
=r\,|e^{i\theta_j}-e^{i\theta}|\\
&\le r\cdot|\theta_j-\theta|<\varepsilon/2
\end{align}
(using $|e^{i\alpha}-e^{i\beta}|\le|\alpha-\beta|$ for all real
$\alpha,\beta$, since
$|e^{i\alpha}-e^{i\beta}|=2|\sin\frac{\alpha-\beta}{2}|\le|\alpha-\beta|$).

Next we compute $f_m(w^*)=(w^*+1)^m-1$:
\[
f_m(w^*)=(w^*+1)^m-1=(r\,e^{i\theta_j})^m-1=r^m e^{im\theta_j}-1.
\]
Since $\theta_j=(2j+1)\pi/m$, we have $m\theta_j=(2j+1)\pi$, so
$e^{im\theta_j}=e^{i(2j+1)\pi}=-1$, and therefore
\[
f_m(w^*)=r^m\cdot(-1)-1=-r^m-1\in\R,\quad -r^m-1<-1\ (\text{since }r>0).
\]
We will fix a small tolerance $\eta>0$ below, after the constant $C$ has been
determined; for any such $\eta$, Theorem~\ref{thm:real} (density of the real
connected domination roots in $(-\infty,0]$) provides a real root
$\rho\in\mathcal{R}_c\cap\R$ with
\[
|\rho-(-r^m-1)|<\eta.
\]

We now apply the inverse function theorem. The derivative of
$f_m(w)=(w+1)^m-1$ at $w^*$ is
\[
f_m'(w^*)=m(w^*+1)^{m-1}=m\,(r\,e^{i\theta_j})^{m-1}
=m\,r^{m-1}e^{i(m-1)\theta_j}.
\]
We have $f_m'(w^*)\neq0$ (since $r>0$ and $|e^{i(m-1)\theta_j}|=1$). By the
local inverse function theorem of complex analysis, $f_m$ is locally
biholomorphic near $w=w^*$: there exist an open neighbourhood $U\ni w^*$ and an
open neighbourhood $V\ni f_m(w^*)=-r^m-1$ such that $f_m:U\to V$ is
biholomorphic and $f_m^{-1}:V\to U$ is holomorphic.

Choose $\eta_0>0$ such that the closed disc
$\overline{B}(-r^m-1,\eta_0)$ is contained in $V$ (possible since $V$ is
open); with $m$ fixed (it was determined by the angular condition above), put
\[
C=\max_{|\zeta-(-r^m-1)|\le\eta_0}\big|(f_m^{-1})'(\zeta)\big|<\infty,
\]
which is finite because $(f_m^{-1})'$ is continuous on the compact disc.
Since the disc is convex, the mean value estimate applied to $f_m^{-1}$
along the segment from $-r^m-1$ to $\rho$ gives
\[
|f_m^{-1}(\rho)-f_m^{-1}(-r^m-1)|\le C\,|\rho-(-r^m-1)|
\qquad\text{for every }\rho\in B(-r^m-1,\eta_0).
\]
Now choose $0<\eta\le\eta_0$ with $C\eta<\varepsilon/2$ (with $C$ already
fixed, this is achieved by taking $\eta$ small); this determines the
tolerance used for $\rho$ above. Since $f_m^{-1}(-r^m-1)=w^*$ and
$|\rho-(-r^m-1)|<\eta$, we have $w=f_m^{-1}(\rho)\in U$ and
\[
|w-w^*|=|f_m^{-1}(\rho)-w^*|\le C\,|\rho-(-r^m-1)|<C\eta<\varepsilon/2.
\]
Therefore we obtain
\[
|w-w_0|\le|w-w^*|+|w^*-w_0|<\varepsilon/2+\varepsilon/2=\varepsilon.
\]
By Lemma~\ref{lem:backward}, $w\in\mathcal{R}_c$ ($w$ is a connected
domination root of $H[K_m]$, where $H$ is a graph corresponding to
$\rho\in\mathcal{R}_c\cap\R$).

Therefore, for every $w_0\in\C$ and $\varepsilon>0$ there exists
$w\in\mathcal{R}_c$ with $|w-w_0|<\varepsilon$, that is,
$\C\subseteq\overline{\mathcal{R}_c}$; the reverse inclusion is trivial, so
$\overline{\mathcal{R}_c}=\C$.\qedhere
\end{proof}

\begin{remark}
The proof follows the same framework that Brown and Tufts~\cite{bt} and Brown
and Beaton~\cite{bb} used for ordinary domination roots: the backward closure
under substitution (Lemma~\ref{lem:backward}) combined with a dense seed set (the
circle of Lemma~\ref{lem:circle} and the real axis of Theorem~\ref{thm:real}).
The substitution formula of Theorem~\ref{thm:cd} is what carries the argument
from $D$ to $D_c$.
\end{remark}

One further consequence, already visible in the cycle seeds,
delimits what kind of root bound one can hope for and is worth recording
explicitly.

\begin{corollary}\label{cor:unbounded}
The set $\mathcal{R}_c$ of connected domination roots is unbounded in $\C$:
for every $M>0$ there exist a connected graph $G$ and a root $z$ of
$D_c(G,x)$ with $|z|>M$. In particular, no graph-independent bound on $|z+1|$
can hold for connected domination roots; any valid bound must grow with the
order of the graph, as does Oboudi's bound
$|z+1|\le\sqrt[\delta+1]{2^n-1}$ for ordinary domination roots~\cite{oboudi}.
\end{corollary}

\begin{proof}
By Lemma~\ref{lem:cycle}, $D_c(C_n,x)=x^{n-2}(x^2+nx+n)$, whose negative root
$z_n^-=\bigl(-n-\sqrt{n^2-4n}\bigr)/2=-n+1+O(1/n)$ tends to $-\infty$ as
$n\to\infty$. Hence $\mathcal{R}_c$ is unbounded in $\C$, and moreover
$|z_n^-+1|\to\infty$. (Unboundedness also follows from
Theorem~\ref{thm:complex}, since a bounded set has bounded closure.)
\end{proof}

\section{Conclusions and further research}\label{sec:concl}

\subsection{Conclusions}

We have determined the distribution of the roots of the connected domination
polynomial. The substitution formula
$D_c(G[K_n],x)=D_c(G,(x+1)^n-1)$ (Theorem~\ref{thm:cd}) reduces the problem to
the construction of seed roots: the closed form for the cycles provides roots
approaching $-1$ from the left and roots diverging to $-\infty$
(Lemma~\ref{lem:cycle}), while the joins $C_m\vee C_n$ provide roots approaching
$-1$ from the right (Lemma~\ref{lem:joinseed}). The backward closure under
substitution then yields the two main results: the closure of the real connected
domination roots is $(-\infty,0]$ (Theorem~\ref{thm:real}), and the closure of
all connected domination roots is the whole complex plane
(Theorem~\ref{thm:complex}). These are, respectively, the connected domination
analogues of the density theorems of Brown--Tufts~\cite{bt} and
Brown--Beaton~\cite{bb} for ordinary domination polynomials. For the independent and
total domination polynomials the substitution behaviour is known~\cite{jahari,dod},
but it provides no spreading mechanism, and the distribution of their roots
remains to be studied (see the open problems below).

\subsection{Directions for further research}

\paragraph{Open Problem 1 (Closure of independent domination roots).} Jahari and
Alikhani proved that $D_i(G[K_n],x)=D_i(G,nx)$~\cite{jahari}, a substitution of
scaling type: under substitution the roots are merely rescaled by $1/n$, so the
spreading mechanism of $(x+1)^n-1$ is not available. What is the closure of the
independent domination roots? Jahari and Alikhani have constructed families of
graphs all of whose roots are real~\cite{jahari}; whether the closure of the real
roots is $(-\infty,0]$ remains unknown.

\paragraph{Open Problem 2 (Closure of total domination roots).} Dod~\cite{dod}
determined the total domination polynomial of the lexicographic product
$G[K_n]$, but the closure of the total domination roots is not clear. Jafari and
Alikhani obtained structural results on the location of total
domination roots~\cite{jafari,jafari2}, but no density statement is known. Since the Brown--Tufts form
fails for total domination (see the discussion in Section~\ref{sec:sub}), the
backward closure mechanism cannot be used; in particular, it is not known whether
the total domination roots are dense in the whole complex plane.

\paragraph{Open Problem 3 (Other domination variants).} The substitution
behaviour of the polynomials attached to paired domination, Roman domination
and Italian (Roman $\{2\}$-) domination is unexplored. In particular, the
interaction between the counting polynomial of Italian domination functions
$f:V\to\{0,1,2\}$ and graph substitution (the lexicographic product) has not
been studied; for the corresponding domination parameters on lexicographic
products we refer to Cabrera Mart\'inez, Estrada-Moreno and
Rodr\'iguez-Vel\'azquez~\cite{cm21}, and to Cabrera Mart\'inez, Garc\'ia-G\'omez
and Rodr\'iguez-Vel\'azquez~\cite{cm22}.

\paragraph{Open Problem 4 (Analytic properties of connected domination roots).}
Oboudi obtained the bound $|z+1|\le\sqrt[\delta+1]{2^n-1}$ for domination
roots~\cite{oboudi}. By Corollary~\ref{cor:unbounded}, no graph-independent
bound can hold for connected domination roots, but a bound depending on the
order (and possibly the minimum degree) of the graph might exist: is there an
explicit bound for connected domination roots analogous to Oboudi's? (For
ordinary domination roots, Omar recently obtained a bound on the modulus that
is linear in the maximum degree~\cite{omar}.) The
unimodality and log-concavity of connected domination polynomials are likewise
open (compare Beaton and Brown~\cite{bb2}), as is the question whether genuine
limit curves, in the sense of the $K_{n,n}$ analysis of~\cite{bt}, can arise
from natural families of connected domination roots; the seed families
constructed here have accumulation points only on the real axis.

\section*{Declarations}

\noindent\textbf{Funding.} The authors declare that no funds, grants, or other
support were received during the preparation of this manuscript.

\medskip
\noindent\textbf{Competing interests.} The authors have no relevant financial
or non-financial interests to disclose.

\medskip
\noindent\textbf{Data availability.} No datasets were generated or analysed
during the current study. All numerical verifications reported here were carried
out by direct enumeration of the relevant subsets and can be reproduced from the
formulas given in the paper.

\medskip
\noindent\textbf{Author contributions.} Pingping Shao: Conceptualization,
Methodology, Formal analysis, Writing~--~original draft. Chengye Zhao:
Methodology, Supervision, Validation, Writing~--~review \& editing. Both authors
approved the final version.

\medskip
\noindent\textbf{Acknowledgements.} During the preparation of this work the
authors used AI-assisted tools, including Doubao, for two purposes: to improve
the readability and language of the manuscript, and to carry out routine
symbolic and numerical checks of the computations reported. The authors
independently verified all such checks, reviewed and edited the content as
needed, and take full responsibility for the content of the publication. No AI
tool is listed as an author.



\begin{thebibliography}{29}

\bibitem{hhs}
Haynes, T.W., Hedetniemi, S.T., Slater, P.J.: Fundamentals of Domination in
Graphs. Pure and Applied Mathematics (New York), Marcel Dekker, New York (1998)

\bibitem{al}
Arocha, J.L., Llano, B.: Mean value for the matching and dominating polynomial.
Discuss. Math. Graph Theory \textbf{20}(1), 57--69 (2000).
\\url{https://doi.org/10.7151/dmgt.1106}

\bibitem{ap}
Alikhani, S., Peng, Y.-H.: Introduction to domination polynomial of a graph.
Ars Combin. \textbf{114}, 257--266 (2014)

\bibitem{aap}
Akbari, S., Alikhani, S., Peng, Y.-H.: Characterization of graphs using
domination polynomials. European J. Combin. \textbf{31}(7), 1714--1724 (2010).
\\url{https://doi.org/10.1016/j.ejc.2010.03.007}

\bibitem{kotek}
Kotek, T., Preen, J., Simon, F., Tittmann, P., Trinks, M.: Recurrence relations
and splitting formulas for the domination polynomial. Electron. J. Combin.
\textbf{19}(3), P47 (2012). \\url{https://doi.org/10.37236/2475}

\bibitem{sokal}
Sokal, A.D.: Chromatic roots are dense in the whole complex plane. Combin.
Probab. Comput. \textbf{13}(2), 221--261 (2004).
\\url{https://doi.org/10.1017/S0963548303006023}

\bibitem{bhn}
Brown, J.I., Hickman, C.A., Nowakowski, R.J.: On the location of roots of
independence polynomials. J. Algebraic Combin. \textbf{19}(3), 273--282 (2004).
\\url{https://doi.org/10.1023/B:JACO.0000030703.39946.70}

\bibitem{bt}
Brown, J.I., Tufts, J.: On the roots of domination polynomials. Graphs Combin.
\textbf{30}(3), 527--547 (2014). \\url{https://doi.org/10.1007/s00373-013-1306-z}

\bibitem{alikhani}
Alikhani, S.: Some new results on domination roots of a graph. Electron. Notes
Discrete Math. \textbf{43}, 425--430 (2013).
\\url{https://doi.org/10.1016/j.endm.2013.07.062}

\bibitem{isrn2013}
Alikhani, S.: On the domination polynomial of some graph operations. ISRN
Combin. \textbf{2013}, Article ID 146595 (2013).
\\url{https://doi.org/10.1155/2013/146595}

\bibitem{oboudi}
Oboudi, M.R.: On the roots of domination polynomial of graphs. Discrete Appl.
Math. \textbf{205}, 126--131 (2016).
\\url{https://doi.org/10.1016/j.dam.2015.12.010}

\bibitem{omar}
Omar, M.: New perspectives on the unimodality of domination polynomials.
Preprint (2026). \\url{https://arxiv.org/abs/2601.14494}

\bibitem{bb}
Brown, J.I., Beaton, I.: On the real roots of domination polynomials. Contrib.
Discrete Math. \textbf{16}(3), 175--182 (2021).
\\url{https://doi.org/10.55016/ojs/cdm.v16i3.72075}

\bibitem{bb2}
Beaton, I., Brown, J.I.: On the unimodality of domination polynomials. Graphs
Combin. \textbf{38}(3), 90 (2022). \\url{https://doi.org/10.1007/s00373-022-02487-x}

\bibitem{ao}
Akbari, S., Oboudi, M.R.: Cycles are determined by their domination
polynomials. Ars Combin. \textbf{116}, 353--358 (2014)

\bibitem{sw}
Sampathkumar, E., Walikar, H.B.: The connected domination number of a graph.
J. Math. Phys. Sci. \textbf{13}(6), 607--613 (1979)

\bibitem{murthy}
Dhananjaya Murthy, B.V., Deepak, G., Soner, N.D.: Connected domination
polynomial of a graph. Int. J. Math. Archive \textbf{4}(11), 90--96 (2013)

\bibitem{mojdeh}
Mojdeh, D.A., Emadi, A.S.: Connected domination polynomial of graphs. Fasc.
Math. \textbf{60}(1), 103--121 (2018).
\\url{https://doi.org/10.1515/fascmath-2018-0007}

\bibitem{yoosuf}
Yoosuf, R., Kuttipulackal, P.: The connected domination polynomial of some
graph constructions. J. Phys. Conf. Ser. \textbf{1850}, 012043 (2021).
\\url{https://doi.org/10.1088/1742-6596/1850/1/012043}

\bibitem{ccc}
Caputi, L., Celoria, D., Collari, C.: Categorifying connected domination via
graph \"uberhomology. J. Pure Appl. Algebra \textbf{227}(9), 107381 (2023).
\\url{https://doi.org/10.1016/j.jpaa.2023.107381}

\bibitem{cdh}
Cockayne, E.J., Dawes, R.M., Hedetniemi, S.T.: Total domination in graphs.
Networks \textbf{10}(3), 211--219 (1980). \\url{https://doi.org/10.1002/net.3230100304}

\bibitem{allan}
Allan, R.B., Laskar, R.: On domination and independent domination numbers of a
graph. Discrete Math. \textbf{23}(2), 73--76 (1978).
\\url{https://doi.org/10.1016/0012-365X(78)90105-X}

\bibitem{jafari}
Jafari, N., Alikhani, S.: On the roots of total domination polynomial of
graphs. J. Discrete Math. Sci. Cryptogr. \textbf{23}(4), 795--807 (2020).
\\url{https://doi.org/10.1080/09720529.2019.1616908}

\bibitem{jahari}
Jahari, S., Alikhani, S.: On the independent domination polynomial of a graph.
Discrete Appl. Math. \textbf{289}, 416--426 (2021).
\\url{https://doi.org/10.1016/j.dam.2020.10.019}

\bibitem{dod}
Dod, M.: Graph products of the trivariate total domination polynomial and
related polynomials. Discrete Appl. Math. \textbf{209}, 92--101 (2016).
\\url{https://doi.org/10.1016/j.dam.2015.10.008}

\bibitem{armada}
Armada, C.L., Canoy, S.R., Jr., Go, C.E.: Forcing subsets for
$\gamma_c$-sets and $\gamma_t$-sets in the lexicographic product of graphs.
Eur. J. Pure Appl. Math. \textbf{12}(4), 1779--1786 (2019).
\\url{https://doi.org/10.29020/nybg.ejpam.v12i4.3485}

\bibitem{jafari2}
Alikhani, S., Jafari, N.: On the roots of total domination polynomial of
graphs, II. Preprint (2019). \\url{https://arxiv.org/abs/1910.05776}

\bibitem{cm21}
Cabrera Mart\'inez, A., Estrada-Moreno, A., Rodr\'iguez-Vel\'azquez, J.A.: From
(secure) $w$-domination in graphs to protection of lexicographic product graphs.
Bull. Malays. Math. Sci. Soc. \textbf{44}(6), 3747--3765 (2021).
\\url{https://doi.org/10.1007/s40840-021-01141-8}

\bibitem{cm22}
Cabrera Mart\'inez, A., Garc\'ia-G\'omez, C., Rodr\'iguez-Vel\'azquez, J.A.:
Perfect domination, Roman domination and perfect Roman domination in
lexicographic product graphs. Fund. Inform. \textbf{185}(3), 201--220 (2022).
\\url{https://doi.org/10.3233/FI-222108}

\end{thebibliography}
\end{document}